\documentclass[12pt]{amsart}

\usepackage{amsmath,amssymb,amsfonts,amsthm,latexsym,graphicx,multirow,hyperref,enumerate}

   \newcommand\Aut{\mathrm{Aut}}
  
\newcommand\C{\mathrm{C}}  \newcommand\calP{\mathcal{P}} \newcommand\Cay{\mathrm{Cay}} \newcommand\Cen{\mathbf{C}}  \newcommand\Cos{\mathsf{Cos}}
 \newcommand\De{\mathit{\Delta}}

 \newcommand\Ga{\Gamma}

\newcommand\K{\mathrm{K}}
 
 \newcommand\Nor{\mathbf{N}}

\newcommand\Si{\Sigma}        \newcommand\Sy{\mathrm{S}} \newcommand\Sym{\mathrm{Sym}} 

 \newcommand\ZZ{\mathbb{Z}}

\newtheorem{theorem}{Theorem}[section]
\newtheorem{lemma}[theorem]{Lemma}
\newtheorem{proposition}[theorem]{Proposition}

\theoremstyle{definition}

\newtheorem{definition}[theorem]{Definition}

\usepackage{xcolor}
\usepackage[normalem]{ulem}

\begin{document}

\title[Arc-transitive circulant digraphs]{An explicit characterization of arc-transitive circulants}

\author[Li]{Cai Heng Li}
\address{Department of Mathematics and SUSTech International Center for Mathematics, Southern University of Science and Technology\\Shenzhen 518055, Guangdong\\P. R. China}
\email{lich@sustech.edu.cn}

\author[Xia]{Binzhou Xia}
\address{School of Mathematics and Statistics\\The University of Melbourne\\Parkville, VIC 3010\\Australia}
\email{binzhoux@unimelb.edu.au}

\author[Zhou]{Sanming Zhou}
\address{School of Mathematics and Statistics\\The University of Melbourne\\Parkville, VIC 3010\\Australia}
\email{sanming@unimelb.edu.au}


\begin{abstract}
A reductive characterization of arc-transitive circulants was given independently by Kov\'acs in 2004 and the first author in 2005.
In this paper, we give an explicit characterization of arc-transitive circulants and their automorphism groups.
Based on this, we give a proof of the fact that arc-transitive circulants are all CI-digraphs.


\end{abstract}

\maketitle

\section{Introduction}

Throughout this paper, a \emph{digraph} is an ordered pair $(V,A)$ with vertex set $V$ and arc set $A$, where $A$ is a set of ordered pairs of elements of $V$.
The cardinality of the vertex set is called the \emph{order} of the digraph.
A digraph $\Ga$ is said to be \emph{arc-transitive} if its automorphism group $\Aut(\Ga)$ acts transitively on the arc set.
A digraph $\Ga$ is called a \emph{circulant} if $\Aut(\Ga)$ has a finite cyclic subgroup that is regular on the vertex set, and is called a \emph{normal circulant} if $\Aut(\Ga)$ has a finite cyclic subgroup that is normal and regular on the vertex set.

Since the work of Chao and Wells~\cite{Chao1971,CW1973} in the 1970's, considerable effort has been made to characterize arc-transitive circulants in the literature.
Remarkable results were achieved under certain conditions such as $2$-arc-transitivity~\cite{ACMX1996,MW2000}, square-free order~\cite{LMM2001}, odd prime-power order~\cite{XBS2004}, small valency~\cite{ABDKM2018} and unit connection set~\cite{ABDKM2018}.
Based on Schur ring and permutation group techniques, a reductive characterization for connected arc-transitive non-normal circulants was obtained independently by Kov\'{a}cs~\cite{Kovacs2004} and the first author~\cite{Li2005}.
Recently, a classification of arc-transitive circulants that are $2$-distance-transitive was given in \cite{CJL2019}.
%

In this paper, we give an explicit characterization (Theorem~\ref{thm1}) of connected arc-transitive circulants, which reveals their structures and determines their automorphism groups.
In fact, Theorem~\ref{thm1} shows that a finite connected arc-transitive circulant can be decomposed into a normal circulant, some complete graphs and an edgeless graph by tensor and lexicographic products.

For digraphs $\Ga$ and $\Si$, their tensor product (direct product) is denoted by $\Ga\times\Si$, and their lexicographic product is denoted by $\Ga[\Si]$ (see Section~\ref{sec1} for the definitions of these products).
For a positive integer $n$, denote by $\Sy_n$ the symmetric group of degree $n$, and denote by $\K_n$ and $\overline\K_n$ the complete graph and the edgeless graph of order $n$, respectively.
Let $\C_4$ denote an undirected cycle of length 4.

\begin{theorem}\label{thm1}
For every connected arc-transitive circulant $\Ga$ of order $n$,
there exist a connected arc-transitive normal circulant $\Ga_0$ of order $n_0$ 
and positive integers $n_1,\dots,n_r,b$, where $r\geqslant0$, such that the following hold:
\begin{enumerate}[{\rm(1)}]
\item $\Ga_0\ncong\C_4$;
\item $n_i\geqslant4$ for $i=1,\dots,r$;
\item $n=n_0n_1\cdots n_rb$, and $n_0,n_1,\dots,n_r$ are pairwise coprime;
\item $\Ga\cong(\Ga_0\times\K_{n_1}\times\dots\times\K_{n_r})[\overline\K_b]$;
\item $\Aut(\Ga)\cong\Sy_b\wr(\Aut(\Ga_0)\times\Sy_{n_1}\times\dots\times\Sy_{n_r})$.
\end{enumerate}
\setcounter{equation}{5}
Moreover, $\Ga$ is uniquely determined by the triple $(\Ga_0,\{n_1,\dots,n_r\},b)$ satisfying the above conditions.
%
\end{theorem}

In view of Theorem~\ref{thm1} we give the following definition.

\begin{definition}
For a connected arc-transitive circulant $\Ga$, a triple $(\Ga_0,\{n_1,\dots,n_r\},b)$ of a finite connected arc-transitive normal circulant $\Ga_0$, a (not necessarily nonempty) set $\{n_1,\dots,n_r\}$ of integers and a positive integer $b$ satisfying conditions~(1)--(5) in Theorem~\ref{thm1} is called a \emph{tensor-lexicographic decomposition} of $\Ga$.
\end{definition}

Here are some remarks on Theorem~\ref{thm1}:

\begin{enumerate}[{\rm(I)}]
\item
The single loop is connected and arc-transitive.
On the other hand, if a connected arc-transitive digraph has order at least two, then it has no loop due to the arc-transitivity.
Thus the circulant $\Ga_0$ in the statement of Theorem~\ref{thm1} is a single loop if $n_0=1$ and has no loop if $n_0\geqslant2$.
In the former case, the tensor product of $\Ga_0$ with any digraph is isomorphic to the digraph, and hence the digraph $\Ga$ described in part~(3) of Theorem~\ref{thm1} is isomorphic to $(\K_{n_1}\times\dots\times\K_{n_r})[\overline\K_b]$.
\item
For each arc-transitive circulant $\Ga_0$ of order $n_0$ and positive integers $n_1,\dots,n_r,b$ such that $n_0,n_1,\dots,n_r$ are pairwise coprime, the digraph $(\Ga_0\times\K_{n_1}\times\dots\times\K_{n_r})[\overline\K_b]$ is an arc-transitive circulant by Lemmas~\ref{lem4} and~\ref{lem10}. The condition $n_i\geqslant4$ as in~(2) of Theorem~\ref{thm1} ensures that $\K_{n_i}$ is a non-normal circulant. If $\Ga_0\cong\C_4$, then since $\C_4\cong\K_2[\overline{\K}_2]$, we deduce from Lemmas~\ref{lem3} and~\ref{lem2} that $(\C_4\times\K_{n_1}\times\dots\times\K_{n_r})[\overline\K_b]\cong(\K_2\times\K_{n_1}\times\dots\times\K_{n_r})[\overline\K_{2b}]$.
\item
To be precise, the statement ``$\Ga$ is uniquely determined by the triple $(\Ga_0,\{n_1,\dots,n_r\},b)$" means that if $(\Ga_0,\{n_1,\dots,n_r\},b)$ and $(\Si_0,\{m_1,\dots,m_s\},c)$ are two tensor-lexicographic decompositions of $\Ga$ then $\Ga_0\cong\Si_0$, $\{n_1,\dots,n_r\}=\{m_1,\dots,m_s\}$ and $b=c$. Theorem~\ref{thm1} shows that every connected arc-transitive circulant has a unique tensor-lexicographic decomposition.
\item
Further descriptions of $\Ga_0$ in Theorem~\ref{thm1} can be found in Section~\ref{sec4} as it is a connected arc-transitive normal circulant.
\item
For any connected arc-transitive circulant $\Ga$, the full automorphism group $\Aut(\Ga)$ is explicitly given by part~(5) of Theorem~\ref{thm1}. However, it is a challenging problem to identify arc-transitive subgroups of $\Aut(\Ga)$, which is equivalent to characterizing certain permutation groups that contain a regular cyclic subgroup; refer to \cite{LP2012}.
\item
Ignoring the orientations of arcs of the circulants, Theorem~\ref{thm1} gives a characterization of edge-transitive undirected circulants.
\item
The existence of tensor-lexicographic decomposition for connected arc-transitive circulants was obtained in~\cite[Theorem~4]{Kovacs2004} by the approach of Schur ring, while the uniqueness was not considered.
\end{enumerate}

Theorem~\ref{thm1} provides an essential tool to study finite arc-transitive circulants. In many situations, the tensor-lexicographic decomposition in Theorem~\ref{thm1} reduces a problem on arc-transitive circulants to the problem on normal ones. We shall illustrate this by verifying the following Theorem~\ref{thm3} on the CI-property of arc-transitive circulants, which was claimed in~\cite[Section~7.3]{Li2002} but did not have any published proof in the literature (to the best of the authors' knowledge).

Given a group $G$ and a nonempty subset $S$ of $G$, the \emph{Cayley digraph} of $G$ with \emph{connection set} $S$, denoted by $\Cay(G,S)$, is the digraph with vertex set $G$ such that a vertex $x$ points to a vertex $y$ if and only if $yx^{-1}\in S$.
It is well known that a digraph $\Ga=(V,A)$ is isomorphic to a Cayley digraph of a group $G$ if and only if
$\Aut(\Ga)$ contains a subgroup that is regular on $V$ and isomorphic to $G$.
Thus a circulant is precisely a Cayley digraph of a finite cyclic group up to isomorphism.
A Cayley digraph $\Cay(G,S)$ is said to be \emph{a CI-digraph} (with respect to $G$) if for each subset $T$ of $G$ with $\Cay(G,S)\cong\Cay(G,T)$ there exists $\alpha\in\Aut(G)$ such that $T=S^\alpha$.
A Cayley digraph of a cyclic group $G$ that is a CI-digraph with respect to $G$ is called a \emph{CI-circulant}.

\begin{theorem}\label{thm3}
Every finite connected arc-transitive circulant is a CI-circulant.
\end{theorem}

The structure of this paper is as follows. In Section~\ref{sec1} we prove the existence of tensor-lexicographic decompositions as described in Theorem~\ref{thm1}. This is then used, together with some result established in Section~\ref{sec4}, to give a proof of Theorem~\ref{thm3} in Section~\ref{sec2}. Finally, in Section~\ref{sec3}, we prove the uniqueness of tensor-lexicographic decompositions, thus completing the proof of Theorem~\ref{thm1}.

\section{Tensor-lexicographic decompositions}\label{sec1}

For digraphs $\Ga=(V_1,A_1)$ and $\Si=(V_2,A_2)$, the \emph{tensor product} (\emph{direct product}) $\Ga\times\Si$
is the digraph with vertex set $V_1\times V_2$ such that $(u_1,u_2)$ points to $(v_1,v_2)$ if and only if $(u_1,v_1)\in A_1$ and $(u_2,v_2)\in A_2$; the \emph{lexicographic product} $\Ga[\Si]$
is the digraph with vertex set $V_1\times V_2$ such that $(u_1,u_2)$ points to $(v_1,v_2)$ if and only if either $(u_1,v_1)\in A_1$, or $u_1=v_1$ and $(u_2,v_2)\in A_2$.
The following lemma follows immediately from the definition of tensor product of digraphs.

\begin{lemma}\label{lem4}
Let $\Ga$ and $\Si$ be digraphs.
Then $\Aut(\Ga\times\Si)\geqslant\Aut(\Ga)\times\Aut(\Si)$, and the following hold:
\begin{enumerate}[{\rm(a)}]
\item if $\Ga$ and $\Si$ are arc-transitive, then so is $\Ga\times\Si$;
\item if $\Ga$ and $\Si$ are circulants of coprime orders, then $\Ga\times\Si$ is a circulant.
\end{enumerate}
\end{lemma}

For a digraph $\Ga=(V,A)$ and a vertex $v\in V$, let $\Ga^-(v)=\{u\in V\mid (u,v)\in A\}$ and $\Ga^+(v)=\{w\in V\mid (v,w)\in A\}$, and call them the \emph{in-neighborhood} and the \emph{out-neighborhood} of $v$ in $\Ga$, respectively.

\begin{lemma}\label{lem10}
Let $\Ga$ be a digraph, and let $b$ be a positive integer.
Then $\Aut(\Ga[\overline\K_b])\geqslant\Sy_b\wr\Aut(\Ga)$, and the following hold:
\begin{enumerate}[{\rm(a)}]
\item if $\Ga$ is arc-transitive, then so is $\Ga[\overline\K_b]$;
\item if $\Ga$ is a circulant, then so is $\Ga[\overline\K_b]$.
\end{enumerate}
\end{lemma}

\begin{proof}
Let $V$ be the vertex set of $\Ga$, and let $W$ be the vertex set of $\overline \K_b$.
Then the vertex set of $\Ga[\overline\K_b]$ is $V\times W$, and for $v_1,v_2\in V$ and $w_1,w_2\in W$, $(v_1,w_1)$ points to $(v_2,w_2)$ in $\Ga[\overline\K_b]$ if and only if $v_1$ points to $v_2$ in $\Ga$.
Write $V=\{v_1,v_2,\dots,v_m\}$, where $m=|V|$, and let $B_i=\{v_i\}\times W$ for $i=1,2,\dots,m$.
For each $\alpha\in\Aut(\Ga)$ and $\beta=(\beta_1,\dots,\beta_m)\in\Sym(B_1)\times\dots\times\Sym(B_m)$, let $(\alpha,\beta)$ act on $V\times W$ by
\[
(v_i,w)^{(\alpha,\beta)}=(v_i^\alpha,w^{\beta_i})\ \text{ for $i\in\{1,\dots,m\}$ and $w\in W$}.
\]
Then $(\alpha,\beta)$ is an automorphism of $\Ga[\overline\K_b]$, and so
\[
\Aut(\Ga[\overline\K_b])\geqslant G:=(\Sym(B_1)\times\dots\times\Sym(B_m))\rtimes\Aut(\Ga)=\Sy_b\wr\Aut(\Ga).
\]

Assume that $\Ga$ is arc-transitive.
Then $\Ga$ is vertex-transitive, and the stabilizer $\Aut(\Ga)_{v_1}$ of $v_1$ in $\Aut(\Ga)$ is transitive on $\Ga^+(v_1)$.
Note that the vertex-transitivity of $\Ga$ implies the vertex-transitivity of $\Ga[\overline\K_b]$.
Moreover, for an arbitrary $w\in W$, the out-neighborhood of $(v_1,w)$ in $\Ga[\overline\K_b]$ is $\Ga^+(v_1)\times W$,
and $(\Sym(B_2)\times\dots\times\Sym(B_m))\rtimes\Aut(\Ga)_{v_1}$ is a subgroup of $\Aut(\Ga[\overline\K_b])$ fixing $(v_1,w)$ and transitive on $\Ga^+(v_1)\times W$.
We conclude that $\Ga[\overline\K_b]$ is arc-transitive.

Now assume that $\Ga$ is a (not necessarily arc-transitive) circulant.
Then there exists $\alpha\in\Aut(\Ga)$ such that $\langle\alpha\rangle$ is regular on $V$.
Take $g\in\Sym(B_1)$ such that $\langle g\rangle$ is regular on $B_1$, and let $\beta=(g,1,\dots,1)\in\Sym(B_1)\times\dots\times\Sym(B_m)$.
Then $(\alpha,\beta)$ is an element of order $mb$ in $\Aut(\Ga[\overline\K_b])$ such that $\langle(\alpha,\beta)\rangle$ is regular on $V\times W$. This shows that $\Ga[\overline\K_b]$ is a circulant, completing the proof.
\end{proof}

In this section we shall prove the existence of tensor-lexicographic decompositions as in Theorem~\ref{thm1}, namely, Proposition~\ref{Structural-thm}.


%

\begin{lemma}\label{lem3}
Let $\Ga$ and $\Si$ be digraphs, and let $m$ be a positive integer. Then $\Ga[\overline\K_m]\times\Si\cong(\Ga\times\Si)[\overline\K_m]$.
\end{lemma}

\begin{proof}
Let $U$ and $W$ be the vertex sets of $\Ga$ and $\Si$, respectively, and let $\De=(V,\emptyset)$ be a digraph such that $\De\cong\overline\K_m$. Consider the bijection
\[
\psi\colon(U\times V)\times W\rightarrow(U\times W)\times V,\quad((u,v),w)\mapsto((u,w),v).
\]
Then $\psi$ is a bijection from the vertex set of $\Ga[\De]\times\Si$ to that of $(\Ga\times\Si)[\De]$.
Moreover, as $\De$ has no arc, for $(u_1,v_1,w_1)$ and $(u_2,v_2,w_2)$ in $U\times V\times W$,
\begin{align*}
&\mbox{$((u_1,v_1),w_1)$ points to $((u_2,v_2),w_2)$ in $\Ga[\De]\times\Si$}\\
\Longleftrightarrow&\mbox{$(u_1,v_1)$ points to $(u_2,v_2)$ in $\Ga[\De]$ and $w_1$ points to $w_2$ in $\Si$}\\
\Longleftrightarrow&\mbox{$u_1$ points to $u_2$ in $\Ga$ and $w_1$ points to $w_2$ in $\Si$}\\
\Longleftrightarrow&\mbox{$(u_1,w_1)$ points to $(u_2,w_2)$ in $\Ga\times\Si$}\\
\Longleftrightarrow&\mbox{$((u_1,w_1),v_1)$ points to $((u_2,w_2),v_2)$ in $(\Ga\times\Si)[\De]$.}
\end{align*}
This shows that $\psi$ is a digraph isomorphism from $\Ga[\De]\times\Si$ to $(\Ga\times\Si)[\De]$, proving the lemma.
\end{proof}

One can easily verify that the lexicographic product of digraphs is associative (see, for example,~\cite[Proposition~5.11]{HIK2011}). Since $\overline\K_m[\overline\K_\ell]\cong\overline\K_{m\ell}$ for all positive integers $m$ and $\ell$, we then have the following lemma.

\begin{lemma}\label{lem2}
Let $\Ga$ be a digraph, and let $m$ and $\ell$ be positive integers.
Then $\Ga[\overline\K_m][\overline\K_\ell]\cong\Ga[\overline\K_{m\ell}]$.
\end{lemma}

For a digraph $\Ga=(V,A)$ and a partition $\calP=\{B_1,\dots,B_m\}$ of $V$, denote by $\Ga/\calP$ the digraph with vertex set $\calP$ such that $(B_i,B_j)$ is an arc if and only if there exist $u\in B_i$ and $v\in B_j$ with $(u,v)\in A$.
If $\calP$ is a partition whose parts are the orbits of some permutation group $N$ on $V$, then $\Ga/\calP$ is also written as $\Ga_N$.
The following result follows from~\cite{EP2002,LM1996,LM1998} (see~\cite[Theorem~2.3]{Li2005}).

\begin{lemma}\label{prop1}
Let $\Ga$ be a connected arc-transitive circulant of order $n$. Then one of the following holds:
\begin{enumerate}[{\rm(a)}]
\item $\Aut(\Ga)=X_1\times\dots\times X_t$, where $t\geqslant1$ and for each $i\in\{1,\dots,t\}$ either $X_i=\Sy_{n_i}$ or $X_i$ has a normal regular cyclic subgroup of order $n_i$, such that $n_1,\dots,n_t$ are pairwise coprime and $n=n_1\cdots n_t$;
\item $\Aut(\Ga)$ has a normal subgroup $N$ such that $\Ga\cong\Ga_N[\overline\K_d]$ for some $d>1$.
\end{enumerate}
\end{lemma}

We also need the following lemma.

\begin{lemma}\label{tensor}
Let $\Ga$ be a $G$-arc-transitive digraph, where $G=X\times Y$ is in product action with permutation groups $X$ and $Y$.
Then $\Ga\cong\Ga_1\times\Ga_2$ for some $X$-arc-transitive digraph $\Ga_1$ and $Y$-arc-transitive digraph $\Ga_2$.
\end{lemma}

\begin{proof}
Let $X$ act on $U$ and $Y$ act on $W$ so that $\Ga$ has vertex set $V:=U\times W$.
For a vertex $v=(u, w)$ of $\Ga$ with $u\in U$ and $w\in W$, the stabilizer of $v$ in $G$ is given by $G_v=X_u\times Y_w$.
Since $\Ga$ is $G$-arc-transitive, there exists $g\in G$ such that $\Ga$ is isomorphic to the coset digraph $\Cos(G,G_v,g)$~\cite[Section~2]{GLX2017}.
Write $g=(x,y)$ with $x\in X$ and $y\in Y$, and let $\Ga_1=\Cos(X,X_u,x)$ and $\Ga_2=\Cos(Y,Y_w,y)$.
Then for each $a=(a_1,a_2)$ and $b=(b_1,b_2)$ in $G=X\times Y$ we have
\begin{align*}
&\mbox{$G_va$ points to $G_vb$ in $\Cos(G,G_v,g)$}\\
\Longleftrightarrow\,&ba^{-1}\in G_vgG_v\\
\Longleftrightarrow\,&\mbox{$(b_1a_1^{-1},b_2a_2^{-1})\in(X_uxX_u)\times(Y_wyY_w)$}\\
\Longleftrightarrow\,&\mbox{$b_1a_1^{-1}\in X_uxX_u$ and $b_2a_2^{-1}\in Y_wyY_w$}\\
\Longleftrightarrow\,&\mbox{$X_ua_1$ points to $X_ub_1$ in $\Ga_1$ and $Y_wa_2$ points to $Y_wb_2$ in $\Ga_2$}\\
\Longleftrightarrow\,&\mbox{$(X_ua_1,Y_wa_2)$ points to $(X_ub_1,Y_wb_2)$ in $\Ga_1\times\Ga_2$.}
\end{align*}
This shows that the (well-defined) map $G_v(a_1,a_2)\mapsto(X_ua_1,H_wa_2)$ is a digraph isomorphism from $\Cos(G,G_v,g)$ to $\Ga_1\times\Ga_2$.
Hence $\Ga\cong\Cos(G,G_v,g)\cong\Ga_1\times\Ga_2$, which completes the proof.
\end{proof}

We are now ready to prove the main result of this section.

\begin{proposition}\label{Structural-thm}
Let $\Ga$ be a connected arc-transitive circulant of order $n$.
Then there exist a connected arc-transitive normal circulant $\Ga_0$ of order $n_0$
and positive integers $n_1,\dots,n_r,b$, where $r\geqslant0$, such that the following hold:
\begin{enumerate}[{\rm(1)}]
\item $\Ga_0\ncong\C_4$;
\item $n_i\geqslant4$ for $i=1,\dots,r$;
\item $n=n_0n_1\cdots n_rb$, and $n_0,n_1,\dots,n_r$ are pairwise coprime;
\item $\Ga\cong(\Ga_0\times\K_{n_1}\times\dots\times\K_{n_r})[\overline\K_b]$;
\item $\Aut(\Ga)\cong\Sy_b\wr(\Aut(\Ga_0)\times\Sy_{n_1}\times\dots\times\Sy_{n_r})$.
\end{enumerate}
\end{proposition}

\begin{proof}
For each normal subgroup $N$ of $\Aut(\Ga)$, the graph $\Ga_N$ is a connected arc-transitive circulant, and so $\Ga_N$ is also described in Lemma~\ref{prop1}.
Let $b$ be the largest integer $d$ such that $\Ga\cong\Si[\overline\K_d]$ for some connected arc-transitive circulant $\Si$.
Then by Lemmas~\ref{lem2} and~\ref{prop1} we have $\Ga\cong\Ga_M[\overline\K_b]$ for some normal subgroup $M$ of $\Aut(\Ga)$ with
\[
\Aut(\Ga_M)=X_1\times\dots\times X_t,
\]
where for each $i\in\{1,\dots,t\}$ either $X_i=\Sy_{n_i}$ or $X_i$ has a normal regular cyclic subgroup of order $n_i$, such that $n_1,\dots,n_t$ are pairwise coprime and $n/b=n_1\cdots n_t$.
Without loss of generality, assume that $\{1,\dots,r\}$ is the subset of $\{1,\dots,t\}$ consisting of elements $i$ such that $X_i$ has no normal regular cyclic subgroup of order $n_i$.
Then $X_i=\Sy_{n_i}$ with $n_i\geqslant4$ for $i\in\{1,\dots,r\}$, and $X_i$ has a normal regular cyclic subgroup of order $n_i$ for $i\in\{r+1,\dots,t\}$.
Let $n_0=n_{r+1}\cdots n_t$ and $X_0=X_{r+1}\times\dots\times X_t$.
It follows that $n/b=n_0n_1\cdots n_r$ with $n_0,n_1,\dots,n_r$ pairwise coprime, $X_0$ has a normal regular cyclic subgroup of order $n_0$, and
\begin{equation}\label{eq1}
\Aut(\Ga_M)=X_0\times X_1\times\dots\times X_r=X_0\times\Sy_{n_1}\times\dots\times\Sy_{n_r}.
\end{equation}
We then derive from Lemma~\ref{tensor} that
\[
\Ga_M\cong\Ga_0\times\K_{n_1}\times\dots\times\K_{n_r}
\]
for some $X_0$-arc-transitive digraph $\Ga_0$.
As a consequence,
\[
\Ga\cong\Ga_M[\overline\K_b]\cong(\Ga_0\times\K_{n_1}\times\dots\times\K_{n_r})[\overline\K_b].
\]
Moreover, $\Ga_0$ is connected as $\Ga_0\times\K_{n_1}\times\dots\times\K_{n_r}\cong\Ga_M$ is connected.
Note that
\[
\Aut(\Ga_M)\cong\Aut(\Ga_0\times\K_{n_1}\times\dots\times\K_{n_r})\geqslant\Aut(\Ga_0)\times\Sy_{n_1}\times\dots\times\Sy_{n_r}.
\]
This together with~\eqref{eq1} and $X_0\leqslant\Aut(\Ga_0)$ implies that
\begin{equation}\label{eq2}
\Aut(\Ga_M)=\Aut(\Ga_0)\times\Sy_{n_1}\times\dots\times\Sy_{n_r}
\end{equation}
and $\Aut(\Ga_0)=X_0$.
Hence $\Aut(\Ga_0)$ has a normal regular cyclic subgroup of order $n_0$, which means that $\Ga_0$ is a normal circulant.
If $\Ga_0\cong\C_4$, then as $\C_4\cong\K_2[\overline\K_2]$, we derive from Lemmas~\ref{lem3} and~\ref{lem2} that
\begin{align*}
\Ga&\cong(\Ga_0\times\K_{n_1}\times\dots\times\K_{n_r})[\overline\K_b]\\
&\cong(\K_2[\overline\K_2]\times\K_{n_1}\times\dots\times\K_{n_r})[\overline\K_b]\\
&\cong(\K_2\times\K_{n_1}\times\dots\times\K_{n_r})[\overline\K_2][\overline\K_b]\\
&\cong(\K_2\times\K_{n_1}\times\dots\times\K_{n_r})[\overline\K_{2b}],
\end{align*}
contradicting our choice of $b$.
Thus $\Ga_0\ncong\C_4$.

Now that we have obtained a connected arc-transitive normal circulant $\Ga_0$ and positive integers $n_1,\dots,n_r,b$ satisfying~(1)--(4), it remains to show that~(5) holds.
Let $\calP$ be the partition of the vertex set of $\Ga$ consisting of orbits of $M$, and let $K$ be the kernel of $\Aut(\Ga)$ acting on $\calP$.
Since $\Ga\cong\Ga_M[\overline\K_b]$, we have $K=\Sy_b^{n/b}$ and $\Aut(\Ga)/K\lesssim\Aut(\Ga_M)$.
Moreover, by Lemma~\ref{lem10} we have
\[
\Aut(\Ga)\gtrsim\Sy_b\wr\Aut(\Ga_M)=\Sy_b^{n/b}\rtimes\Aut(\Ga_M).
\]
Hence $\Aut(\Ga)\cong\Sy_b\wr\Aut(\Ga_M)$, which then together with~\eqref{eq2} leads to~(5), as required.
\end{proof}

\section{Normal circulants}\label{sec4}

Let $G$ be a group. For $g\in G$ denote by $\widehat{g}$ the permutation of $G$ such that $x^{\widehat{g}}=xg$ for each $x\in G$.
Then $\widehat{G}:=\{\widehat{g}\mid g\in G\}$ is a regular permutation group on $G$.
For a subset $S$ of $G$, let
\[
\Aut(G,S)=\{\alpha\in\Aut(G)\mid S^\alpha=S\}.
\]
Then $\widehat{G}\rtimes\Aut(G,S)$ is a subgroup of $\Nor_{\Aut(\Ga)}(\widehat{G})$, the normalizer of $\widehat{G}$ in $\Aut(\Ga)$, where $\Ga=\Cay(G,S)$.
In fact, we have the following:

\begin{lemma}\label{prop3}
\emph{(\cite{Godsil1981,Xu1998})}
Let $\Ga=\Cay(G,S)$ be a Cayley digraph of a finite group $G$. Then $\Nor_{\Aut(\Ga)}(\widehat{G})=\widehat{G}\rtimes\Aut(G,S)$.
\end{lemma}

By Lemma~\ref{prop3} we see that if $\Ga=\Cay(G,S)$ is a connected arc-transitive normal Cayley digraph, then $\Aut(\Ga)=\widehat{G}\rtimes\Aut(G,S)$, whose stabilizer of the vertex $1$ is $\Aut(G,S)$.
This implies the next lemma.

\begin{lemma}\label{lem9}
Let $\Ga=\Cay(G,S)$ be a connected arc-transitive normal Cayley digraph of a finite cyclic group $G$. Then $S$ consists of generators of $G$, and $\Aut(G,S)$ acts regularly on $S$.
\end{lemma}


The ``only if'' part of the next result is in fact a consequence of Proposition~\ref{prop4}, Lemma~\ref{lem9} and the assertion of Toida's conjecture~\cite{Toida1977} that every circulant with connection set consisting of generators is CI, which was proved independently in~\cite{DM2002} and~\cite{KMP2001}. Here we give a direct and self-contained proof of this lemma. For a prime number $r$ and a finite cyclic group $X$, denote the unique Sylow $r$-subgroup and the unique Hall $r'$-subgroup of $X$ by $X_r$ and $X_{r'}$, respectively.

\begin{proposition}\label{lem12}
A finite connected arc-transitive circulant is normal if and only if its automorphism group contains a unique regular cyclic subgroup.
\end{proposition}

\begin{proof}
Let $\Ga=\Cay(G,S)$ be a connected arc-transitive Cayley digraph of a finite cyclic group $G$. If $\Aut(\Ga)$ contains a unique regular cyclic group, then every conjugate of $\widehat{G}$ in $\Aut(\Ga)$ is equal to $\widehat{G}$, and so $\widehat{G}$ is normal in $\Aut(\Ga)$.
Conversely, suppose that $\widehat{G}$ is normal in $\Aut(\Ga)$. Write $N=\widehat{G}$. Then since $N$ is an abelian regular subgroup of $\Aut(\Ga)$, we obtain $\Cen_{\Aut(\Ga)}(N)=N$. By Lemma~\ref{prop3} we have $\Aut(\Ga)=N\rtimes\Aut(G,S)$.

Let $H$ be a regular cyclic subgroup of $\Aut(\Ga)$. Suppose towards a contradiction that $H\neq N$. Take $p$ to be the largest prime number such that $H_p\neq N_p$ and take $h$ to be an arbitrary element of $H_p\setminus N_p$. Then $h\notin N=\Cen_{\Aut(\Ga)}(N)$. Consider an arbitrary prime divisor $q$ of $|N|$ such that $h\notin\Cen_{\Aut(\Ga)}(N_q)$. Since $h\in H_p$ and $h$ induces a nontrivial automorphism of $N_q\cong\ZZ_{|N_q|}$, we deduce that $p$ divides $|\Aut(\ZZ_{|N_q|})|$. Accordingly, $p\leqslant q$. If $p<q$, then $H_q=N_q$ and hence
\[
h\in H\leqslant\Cen_{\Aut(\Ga)}(H_q)=\Cen_{\Aut(\Ga)}(N_q),
\]
a contradiction. Thus $p=q$, which shows that $h\in\Cen_{\Aut(\Ga)}(N_{p'})$. Now $h$ is an element of $\Aut(\Ga)=N\rtimes\Aut(G,S)$ with order a $p$-power and conjugation action trivial on $N_{p'}$ but nontrivial on $N$. Write $h=k\alpha$ with $k\in N$ and $\alpha\in\Aut(G,S)$. Since for each $g\in G$ the conjugate of $\widehat{g}$ by $h=k\alpha$ is $\widehat{g^\alpha}$, it follows that $\alpha$ is a nonidentity element with order a $p$-power such that $g^\alpha=g$ for all $g\in G_{p'}$. Since $\Aut(G)=\Aut(G_p)\times\Aut(G_{p'})$, we may then write $\alpha=(\beta,1)$ with $\beta\in\Aut(G_p)$ such that $\beta$ has order $p^\ell$ for some positive integer $\ell$. Let $|G_p|=p^m$. Then we have $m>1$.

Suppose $p>2$. Let $\gamma$ be the automorphism of $G_p$ sending $g$ to $g^{p^{m-1}+1}$ for all $g\in G_p$. Then $\gamma$ has order $p$. Since $\Aut(G_p)$ is cyclic and $\beta$ is an element of $\Aut(G_p)$ of order $p^\ell$, it follows that $\gamma\in\langle\beta\rangle$ and so $(\gamma,1)\in\langle\alpha\rangle\leqslant\Aut(G,S)$. Take an arbitrary $s\in S$ and write $s=ab$ with $a\in G_p$ and $b\in G_{p'}$. By Lemma~\ref{lem9} we know that $s$ generates $G$, whence $a$ generates $G_p$. For $j\in\{0,1,\dots,p-1\}$, as $(p^{m-1}+1)^j\equiv jp^{m-1}+1\pmod{p^m}$, we derive that
\[
sa^{jp^{m-1}}=a^{jp^{m-1}+1}b=a^{(p^{m-1}+1)^j}b=a^{\gamma^j}b=s^{(\gamma,1)^j}\in s^{\langle\alpha\rangle}\subseteq S.
\]
Hence $sP\subseteq S$, where $P=\langle a^{p^{m-1}}\rangle$ is the unique subgroup of order $p$ in $G$. This shows that $S$ is a union of cosets of $P$ in $G$. Hence $\Ga=\Si[\overline\K_p]$ for some digraph $\Si$, and it follows that $\Aut(\Ga)\geqslant\Sy_p\wr\Aut(\Si)$.
This implies that $N=\widehat{G}$ is not normal in $\Aut(\Ga)$, a contradiction.

Now $p=2$. Let $\sigma$ and $\delta$ be the automorphisms of $G_2$ such that $g^\sigma=g^{-1}$ and $g^\delta=g^5$ for all $g\in G_2$. Note that $\Aut(G_2)=\langle\sigma\rangle\times\langle\delta\rangle\cong\ZZ_2\times\ZZ_{2^{m-2}}$ if $m\geqslant3$ and $\Aut(G_2)=\langle\sigma\rangle\cong\ZZ_2$ if $m=2$. Suppose that $m\geqslant3$ and $\delta^{2^{m-3}}\in\langle\beta\rangle$. Then $(\delta^{2^{m-3}},1)\in\langle\alpha\rangle\leqslant\Aut(G,S)$. Take an arbitrary $s\in S$ and write $s=ab$ with $a\in G_2$ and $b\in G_{2'}$. As Lemma~\ref{lem9} asserts, $s$ generates $G$, whence $a$ generates $G_2$. Note that $(5^{2^{m-3}}-1)_2=2^{m-1}$. Since
\[
sa^{5^{2^{m-3}}-1}=a^{5^{2^{m-3}}}b=a^{\delta^{2^{m-3}}}b=s^{(\delta^{2^{m-3}},1)}\in s^{\langle\alpha\rangle}\subseteq S,
\]
we then obtain $sP\subseteq S$, where $P=\langle a^{2^{m-1}}\rangle$ is the unique subgroup of order $2$ in $G$. This shows that $S$ is a union of cosets of $P$ in $G$. Hence $\Ga=\Si[\overline\K_2]$ for some digraph $\Si$, and it follows that $\Aut(\Ga)\geqslant\Sy_2\wr\Aut(\Si)$. This implies that $\Si\cong\K_2$ as $N=\widehat{G}$ is normal in $\Aut(\Ga)$ (see Lemma~\ref{lem5} below). However, it follows that $\Ga\cong\C_4$, in which case $\Aut(\Ga)$ has a unique regular cyclic subgroup, contradicting our assumption that $H\neq N$. Since the only nontrivial subgroups of $\Aut(G_2)$ for $m\geqslant3$ that do not contain $\delta^{2^{m-3}}$ are $\langle\sigma\rangle$ and $\langle\sigma\delta^{2^{m-3}}\rangle$, we then conclude that one of the following holds:
\begin{enumerate}[{\rm(i)}]
\item $m\geqslant2$ and $H_2\leqslant N_2\rtimes\langle\sigma\rangle$;
\item $m\geqslant3$ and $H_2\leqslant N_2\rtimes\langle\sigma\delta^{2^{m-3}}\rangle$.
\end{enumerate}
Let $\rho$ be a generator of $N_2$. Note that $N_2\cong G_2\cong\ZZ_{2^m}$.

First assume that~(i) occurs. Then since $H_2$ and $N_2$ are both cyclic subgroups of index $2$ in the dihedral group $N_2\rtimes\langle\sigma\rangle$, we have $H_2=N_2$, a contradiction.

Next assume that~(ii) occurs. Since $H_2\cong N_2$ and $H_2\neq N_2$, we see that $H_2\nleqslant N_2$. As $|N_2\rtimes\langle\sigma\delta^{2^{m-3}}\rangle|=2|N_2|$, this implies that $H_2N_2=N_2\rtimes\langle\sigma\delta^{2^{m-3}}\rangle$ and $H_2\cap N_2=\langle\rho^2\rangle$. Note that $H_2\cap N_2$ is centralized by both $H_2$ and $N_2$. It follows that $H_2\cap N_2$ is centralized by $H_2N_2=N_2\rtimes\langle\sigma\delta^{2^{m-3}}\rangle$. In particular,  $\rho^2$ commutes with $\sigma\delta^{2^{m-3}}$. This means that $\rho^{-2\cdot5^{2^{m-3}}}=\rho^2$, which is equivalent to $2(5^{2^{m-3}}+1)\equiv0\pmod{2^m}$. However, since $5^{2^{m-3}}+1\equiv2\pmod{4}$, we derive that $2(5^{2^{m-3}}+1)$ is not divisible by $8$ and hence not divisible by $2^m$ as $m\geqslant3$, a contradiction. The proof is thus completed.
\end{proof}

\section{CI-property of arc-transitive circulants}\label{sec2}

In this section we prove that every finite connected arc-transitive circulant is a CI-digraph, thus verifying Theorem~\ref{thm3}.

\begin{lemma}\label{prop4}
\emph{(\cite{Babai1977})} Let $\Ga$ be a Cayley digraph of a group $G$. Then $\Cay(G,S)$ is a CI-digraph if and only if each subgroup of $\Aut(\Ga)$ conjugate to $\widehat{G}$ in $\Sym(G)$ is conjugate to $\widehat{G}$ in $\Aut(\Ga)$.
\end{lemma}

%

The following lemma is also needed in this section.

\begin{lemma}\label{lem11}
Let $G$ be a cyclic transitive subgroup of $\Sym(\Omega)\times\Sym(\Delta)$ acting on $\Omega\times\Delta$ by product action. Then there exist regular cyclic subgroups $H$ and $K$ of $\Sym(\Omega)$ and $\Sym(\Delta)$, respectively, such that $G=H\times K$. In particular, $|\Omega|$ and $|\Delta|$ are coprime.
\end{lemma}

\begin{proof}
Let $H$ and $K$ be the projections of $G$ to $\Sym(\Omega)$ and $\Sym(\Delta)$, respectively. Then $G\leqslant H\times K$. Since $G$ is a cyclic subgroup transitive on $\Omega\times\Delta$ by product action, we see that $H$ is a cyclic transitive subgroup of $\Sym(\Omega)$ and $K$ is a cyclic transitive subgroup of $\Sym(\Delta)$. Note that every abelian transitive permutation group is regular. It follows that $|H|=|\Omega|$ and $|K|=|\Delta|$. Then since $G$ is transitive on $\Omega\times\Delta$, we derive from $G\leqslant H\times K$ that $G=H\times K$.
\end{proof}

For a digraph $\Ga$, denote by $V(\Ga)$ the vertex set of $\Ga$.

\begin{proof}[Proof of Theorem~$\ref{thm3}$] Let $\Ga=\Cay(G,S)$ be a connected arc-transitive digraph of a finite cyclic group $G$. By virtue of Lemma~\ref{prop4} it suffices to prove that each regular cyclic subgroup $H$ of $\Aut(\Ga)$ is conjugate to $\widehat{G}$.

By Proposition~\ref{Structural-thm}, $\Ga$ has a tensor-lexicographic decomposition $(\Ga_0,\{n_1,\dots,n_r\},b)$. Let $\Si=\Ga_0\times \K_{n_1}\times\dots\times \K_{n_r}$ and $m=|V(\Si)|$. Then we may regard $\Ga=\Si[\overline\K_b]$ so that $\Aut(\Ga)=\Sy_b\wr\Aut(\Si)$ with $\Aut(\Si)=\Aut(\Ga_0)\times\Sy_{n_1}\times\dots\times\Sy_{n_r}$. Let $\varphi$ be the projection of $\Aut(\Ga)=\Sy_b\wr\Aut(\Si)$ onto $\Aut(\Si)$ and let $N=\Sy_b^m$ be the kernel of $\varphi$. Since $\widehat{G}$ is a cyclic transitive permutation group on $V(\Ga)$, the image $\varphi(\widehat{G})$ of $\widehat{G}$ under $\varphi$ is a cyclic transitive permutation group on $V(\Si)$ and hence is regular. Similarly, $\varphi(H)$ is also a regular cyclic subgroup of $\Aut(\Si)$. Note that $\Aut(\Si)=\Aut(\Ga_0)\times\Sy_{n_1}\times\dots\times\Sy_{n_r}$ with $|V(\Ga_0)|,n_1,\dots,n_r$ pairwise coprime. Appealing to Lemma~\ref{lem11} we obtain
\[
\varphi(H)=A_0\times A_1\times\dots\times A_r\quad\text{and}\quad\varphi(\widehat{G})=B_0\times B_1\times\dots\times B_r
\]
such that $A_0$ and $B_0$ are regular cyclic subgroups of $\Aut(\Ga_0)$ and $A_i$ and $B_i$ are regular cyclic subgroups of $\Sy_{n_i}$ for $i=1,\dots,r$. Since $\Ga_0$ is a connected arc-transitive normal circulant, Proposition~\ref{lem12} implies that $A_0=B_0$. Moreover, $A_i$ and $B_i$ are conjugate in $\Sy_{n_i}$ for $i=1,\dots,r$. Thus $\varphi(H)$ and $\varphi(\widehat{G})$ are conjugate in $\Aut(\Ga_0)\times\Sy_{n_1}\times\dots\times\Sy_{n_r}=\Aut(\Si)$, and so there exist $h\in H$, $\alpha\in\widehat{G}$ and $\sigma\in\Aut(\Ga)$ such that $H=\langle h\rangle$, $\widehat{G}=\langle\alpha\rangle$ and $\varphi(\alpha)=\varphi(h^\sigma)$.

Since $h$ and $\alpha$ are $(mb)$-cycles, we may write $\alpha=(\alpha_1,\dots,\alpha_{mb})$ and $h^\sigma=(\beta_1,\dots,\beta_{mb})$ with $V(\Ga)=\{\alpha_1,\dots,\alpha_{mb}\}=\{\beta_1,\dots,\beta_{mb}\}$ and $\alpha_1=\beta_1$. For a vertex $w=(u,v)$ of $\Ga=\Si[\overline\K_b]$, where $u\in V(\Si)$ and $v\in V(\overline\K_b)$, denote $\overline{w}=u$. As $\langle\varphi(\alpha)\rangle=\varphi(\langle\alpha\rangle)=\varphi(H)$ is a regular cyclic subgroup of $\Aut(\Si)$, the permutation $\varphi(\alpha)$ of $V(\Si)$ is an $m$-cycle. Thereby we derive from the conditions $\alpha_1=\beta_1$ and $\varphi(\alpha)=\varphi(h^\sigma)$ that $\overline{\alpha_{i+mj}}=\overline{\beta_{i+mj}}$ for all $i\in\{1,\dots,m\}$ and $j\in\{0,1,\dots,b-1\}$. This shows that
\[
\{\beta_i,\beta_{i+m}\dots,\beta_{i+m(b-1)}\}=\{\alpha_i,\alpha_{i+m}\dots,\alpha_{i+m(b-1)}\}=\overline{\alpha_i}\times V(\overline\K_b)
\]
for $i=1,\dots,m$, and so these are the $m$ orbits of $N$ on $V(\Ga)$. Since $N=\Sy_b^m$, there exists $\delta\in N$ such that $\beta_{i+mj}^\delta=\alpha_{i+mj}$ for all $i\in\{1,\dots,m\}$ and $j\in\{0,1,\dots,b-1\}$. Hence
\[
h^{\sigma\delta}=(\beta_1,\dots,\beta_{mb})^\delta=(\beta_1^\delta,\dots,\beta_{mb}^\delta)=(\alpha_1,\dots,\alpha_{mb})=\alpha.
\]
It follows that $H^{\sigma\delta}=\langle h\rangle^{\sigma\delta}=\langle h^{\sigma\delta}\rangle=\langle\alpha\rangle=\widehat{G}$, which means that $H$ is conjugate to $\widehat{G}$ in $\Aut(\Ga)$, as desired.
\end{proof}

\section{Proof of Theorem~\ref{thm1}}\label{sec3}


\begin{definition}
A digraph $\Ga$ is said to be \emph{$R$-thick} if it has distinct vertices $u$ and $v$ such that $\Ga^+(u)=\Ga^+(v)$ and $\Ga^-(u)=\Ga^-(v)$.
A digraph that is not $R$-thick is said to be \emph{$R$-thin}.
\end{definition}

\begin{lemma}\label{lem5}
Let $\Ga$ be an $R$-thick vertex-transitive digraph with no loop.
Then the following hold:
\begin{enumerate}[{\rm(a)}]
\item $\Ga\cong\Si[\overline\K_b]$ for some digraph $\Si$ and $b\geqslant2$;
\item if $\Ga$ is a normal Cayley digraph, then $\Ga$ has order at most $4$;
\item if $\Ga$ is a normal circulant with nonempty arc set, then $\Ga\cong\C_4$.
\end{enumerate}
\end{lemma}

\begin{proof}
Let $\Ga=(V,A)$, and let $G=\Aut(\Ga)$.
For $v\in V$, let
\[
B(v)=\{u\in V \mid \Ga^+(u)=\Ga^+(v),\ \Ga^-(u)=\Ga^-(v)\}.
\]
Then the sets $B(v)$ with $v\in V$ form a $G$-invariant partition $\calP=\{B_1,B_2,\dots,B_k\}$ of $V$.
Since $G$ is transitive on $V$ and $\Ga$ is $R$-thick, we have $|B_1|=|B_2|=\dots=|B_k|=b$ for some integer $b\geqslant2$.
Since $\Ga$ has no loop, the $R$-thickness of $\Ga$ implies that there is no arc between vertices in the same part $B_i$ for any $i$.
For $i,j\in\{1,2,\dots,k\}$, if there is an arc $(u,v)$ of $\Ga$ with $u\in B_i$ and $v\in B_j$, then as $\Ga$ is $R$-thick, each vertex in $B_i$ points to all vertices of $B_j$.
This shows that $\Ga\cong\Si[\overline\K_b]$, where $\Si=\Ga/\calP$, proving part~(a).

To prove part~(b), suppose for a contradiction that $\Ga$ is a normal Cayley digraph of a group $G$ with $|G|\geqslant5$.
Take $w\in G$ such that $1$ and $w$ are distinct vertices in the same part of the partition $\calP$.
Then the transposition $\tau$ swapping $1$ and $w$ is an automorphism of $\Ga$.
Take $x\in G\setminus\{1,w\}$.
As
\[
|x(G\setminus\{1,w\})|=|G\setminus\{1,w\}|=|G|-2\geqslant3,
\]
there exists $y\in G\setminus\{1,w\}$ such that $xy\notin\{1,w\}$.
Then $x$, $y$ and $xy$ are all fixed by $\tau$.
Since $\Ga$ is a normal Cayley digraph of $G$, Lemma~\ref{prop3} implies that the automorphism $\tau$ of $\Ga$ can be written as $\tau=\overline{g}\alpha$, where $\overline{g}$ is the right multiplication by $g\in G$ and $\alpha\in\Aut(G)$.
It follows that
\[
\tau\overline{g}=\tau^{-1}\overline{g}=\alpha^{-1}\in\Aut(G),
\]
and hence $(xy)^{\tau\overline{g}}=x^{\tau\overline{g}}y^{\tau\overline{g}}$.
However,
\[
(xy)^{\tau\overline{g}}=(xy)^{\overline{g}}=xyg,\quad x^{\tau\overline{g}}=x^{\overline{g}}=xg\quad\text{and}\quad y^{\tau\overline{g}}=y^{\overline{g}}=yg,
\]
which yields $xyg=(xg)(yg)$, or equivalently, $g=1$.
This implies that $\tau=\alpha$ and so
\[
w=1^\tau=1^\alpha=1,
\]
a contradiction. Thus part~(b) holds.

Finally, assume that $\Ga$ is a normal circulant of order $n$ with nonempty arc set.
Then $n=kb$ with $b\geqslant2$.
If $k=1$, then $\Ga\cong\overline\K_b$ has an empty arc set, a contradiction.
Thus $k\geqslant2$ and hence $n=kb\geqslant4$.
This together with part~(b) gives $n=4$.
As $\Ga$ is an $R$-thick normal circulant with no loop, we then conclude that $\Ga\cong\C_4$, as part~(c) asserts.
\end{proof}

\begin{lemma}\label{lem13}
Let $\Ga_1$ and $\Ga_2$ be $R$-thin digraphs with no loops and let $b_1$ and $b_2$ be positive integers such that $\Ga_1[\overline\K_{b_1}]\cong\Ga_2[\overline\K_{b_2}]$. Then $\Ga_1\cong\Ga_2$ and $b_1=b_2$.
\end{lemma}

\begin{proof}
For $i=1,2$, define a binary relation $R_i$ on $V(\Ga_i[\overline\K_{b_i}])$ by letting $(g,h)\in R_i$ if and only if
\[
(\Ga_i[\overline\K_{b_i}])^+(g)=(\Ga_i[\overline\K_{b_i}])^+(h)\quad\text{and}\quad
(\Ga_i[\overline\K_{b_i}])^-(g)=(\Ga_i[\overline\K_{b_i}])^-(h),
\]
where $g$ and $h$ are arbitrary vertices of $\Ga_i[\overline\K_{b_i}]$. Then $R_i$ is an equivalence relation and so can be viewed as a partition on $V(\Ga_i[\overline\K_{b_i}])$. Write $g=(u,x)$ and $h=(v,y)$ with $u,v\in V(\Ga_i)$ and $x,y\in V(\overline\K_{b_i})$. It is clear that
\[
(\Ga_i[\overline\K_{b_i}])^\varepsilon(g)=\Ga_i^\varepsilon(u)\times V(\overline\K_{b_i})\quad\text{and}\quad
(\Ga_i[\overline\K_{b_i}])^\varepsilon(h)=\Ga_i^\varepsilon(v)\times V(\overline\K_{b_i})
\]
for $\varepsilon=+,-$. Therefore, $(g,h)\in R_i$ if and only if $\Ga_i^+(u)=\Ga_i^+(v)$ and $\Ga_i^-(u)=\Ga_i^-(v)$. Since $\Ga_i$ is $R$-thin and has no loop, this shows that $(g,h)\in R_i$ if and only if $u=v$. Hence the equivalence classes of $R_i$ are $\{u\}\times V(\overline\K_{b_i})$ with $u$ running over $V(\Ga_i)$. In particular, each equivalence class of $R_i$ has size $b_i$ and $(\Ga_i[\overline\K_{b_i}])/R_i\cong\Ga_i$. As $\Ga_1[\overline\K_{b_1}]\cong\Ga_2[\overline\K_{b_2}]$, we then derive that $b_1=b_2$ and
\[
\Ga_1\cong(\Ga_1[\overline\K_{b_1}])/R_1\cong(\Ga_2[\overline\K_{b_2}])/R_2\cong\Ga_2.
\]
The proof is thus completed.
\end{proof}

\begin{lemma}\label{lem14}
Let $\Ga_0$ and $\Si_0$ be normal circulants and let $n_1,\dots,n_r,m_1,\dots,m_s$ be integers greater than $3$ such that $|V(\Ga_0)|,n_1,\dots,n_r$ are pairwise coprime, $|V(\Si_0)|,m_1,\dots,m_s$ are pairwise coprime and $\Aut(\Ga_0)\times\Sy_{n_1}\times\dots\times\Sy_{n_r}\cong\Aut(\Si_0)\times\Sy_{m_1}\times\dots\times\Sy_{m_s}$. Then $r=s$ and $\{n_1,\dots,n_r\}=\{m_1,\dots,m_s\}$.
\end{lemma}

\begin{proof}
Take an arbitrary $i\in\{1,\dots,r\}$. Let $G=G_0\times G_1\times\dots\times G_r$ with $G_0\cong\Aut(\Ga_0)$ and $G_i\cong\Sy_{n_i}$ for $i=1,\dots,r$ and $H=H_0\times H_1\times\dots\times H_s$ with $H_0\cong\Aut(\Si_0)$ and $H_j\cong\Sy_{m_j}$ for $j=1,\dots,s$. Since $G_i$ is normal in $G$ and $G\cong H$, there exists a normal subgroup $X$ of $H$ such that $X\cong\Sy_{n_i}$. It follows that $X\cap H_s$ is a normal subgroup of both $X\cong\Sy_{n_i}$ and $H_s\cong\Sy_{m_s}$. As $n_i$ and $m_s$ are both greater than $3$, we then deduce that either $n_i=m_s$ or $X\cap H_s=1$. If $X\cap H_s=1$, then $H/H_s\cong H_0\times H_1\times\dots\times H_{s-1}$ has a normal subgroup $XH_s/H_s\cong X\cong\Sy_{n_i}$. Applying the above argument inductively we obtain that either $n_i\in\{m_1,\dots,m_s\}$ or $\Aut(\Si_0)$ has a normal subgroup isomorphic to $\Sy_{n_i}$. Suppose that the latter occurs. Since $\Si_0$ is a normal circulant, Lemma~\ref{prop3} shows that $\Aut(\Si_0)$ is an extension of a cyclic group by an abelian group and hence every subgroup of $\Aut(\Si_0)$ is an extension of a cyclic group by an abelian group. However, $\Sy_{n_i}$ cannot be such an extension as $n_i>3$, a contradiction. Thus $n_i\in\{m_1,\dots,m_s\}$. Since $i$ is an arbitrary element of $\{1,\dots,r\}$, we then have $\{n_1,\dots,n_r\}\subseteq\{m_1,\dots,m_s\}$ and similarly $\{m_1,\dots,m_s\}\subseteq\{n_1,\dots,n_r\}$. This shows that $r=s$ and $\{n_1,\dots,n_r\}=\{m_1,\dots,m_s\}$, as desired.
\end{proof}

We are now in a position to prove Theorem~\ref{thm1}.

\begin{proof}[Proof of Theorem~$\ref{thm1}$] By Proposition~\ref{Structural-thm}, we only need to prove the uniqueness of tensor-lexicographic decompositions. Without loss of generality, assume $|V(\Ga)|>1$. Let $(\Gamma_0,N,b)$ and $(\Sigma_0,M,c)$ be tensor-lexicographic decompositions of $\Gamma$, where $N=\{n_1,\dots,n_r\}$ and $M=\{m_1,\dots,m_s\}$. Then $|V(\Ga_0)|,n_1,\dots,n_r$ are pairwise coprime, and $|V(\Si_0)|,m_1,\dots,m_s$ are pairwise coprime. Let $\Ga_1=\Ga_0\times \K_{n_1}\times\dots\times \K_{n_r}$ and $\Ga_2=\Si_0\times \K_{m_1}\times\dots\times \K_{m_s}$. Then
\[
\Ga\cong\Ga_1[\overline\K_b]\cong\Ga_2[\overline\K_c],
\]
and $\Ga_1$ and $\Ga_2$ are both arc-transitive with $\Aut(\Ga_1)=\Aut(\Ga_0)\times\Sy_{n_1}\times\dots\times\Sy_{n_r}$ and $\Aut(\Ga_2)=\Aut(\Si_0)\times\Sy_{m_1}\times\dots\times\Sy_{m_s}$. It follows from Lemma~\ref{lem5}(c) that $\Ga_0$ and $\Si_0$ are $R$-thin. Hence $\Ga_1=\Ga_0\times \K_{n_1}\times\dots\times \K_{n_r}$ and $\Ga_2=\Si_0\times \K_{m_1}\times\dots\times \K_{m_s}$ are $R$-thin. Since $\Ga$ is connected, neither $|V(\Ga_1)|$ nor $|V(\Ga_2)|$ is equal to $1$. Hence neither $\Ga_1$ nor $\Ga_2$ has a loop. By Lemma~\ref{lem13} we derive from $\Ga_1[\overline\K_b]\cong\Ga_2[\overline\K_c]$ that $\Ga_1\cong\Ga_2$ and $b=c$. Then by Lemma~\ref{lem14} we deduce that $N=M$. To complete the proof, it remains to show $\Ga_0\cong\Si_0$. Since $b=c$ and $N=M$, we have
\[
|V(\Ga_0)|=\frac{|V(\Ga_1)|}{n_1\dots,n_r}=\frac{|V(\Ga)|}{n_1\dots,n_rb}
=\frac{|V(\Ga)|}{m_1\dots,m_sc}=\frac{|V(\Ga_2)|}{m_1\dots,m_s}=|V(\Si_0)|.
\]
Let $n_0=|V(\Ga_0)|=|V(\Si_0)|$. Then there exist generating subsets $S$ and $T$ of $\ZZ_{n_0}$ such that $\Ga_0=\Cay(\ZZ_{n_0},S)$ and $\Si_0=\Cay(\ZZ_{n_0},T)$. It follows that
\begin{align*}
&\Cay(\ZZ_{n_0}\times\ZZ_{n_1}\times\dots\times\ZZ_{n_r},S\times(\ZZ_{n_1}\setminus\{0\})\times\dots\times(\ZZ_{n_r}\setminus\{0\}))\\
\cong&\Cay(\ZZ_{n_0},S)\times\Cay(\ZZ_{n_1},\ZZ_{n_1}\setminus\{0\})\times\dots\times\Cay(\ZZ_{n_r},\ZZ_{n_r}\setminus\{0\})\\
\cong&\Ga_0\times \K_{n_1}\times\dots\times \K_{n_r}\\
=&\Ga_1
\end{align*}
is a connected arc-transitive circulant and thus a CI-digraph by Theorem~\ref{thm3}, and
\begin{align*}
&\Cay(\ZZ_{n_0}\times\ZZ_{n_1}\times\dots\times\ZZ_{n_r},T\times(\ZZ_{n_1}\setminus\{0\})\times\dots\times(\ZZ_{n_r}\setminus\{0\}))\\
\cong&\Cay(\ZZ_{n_0},T)\times\Cay(\ZZ_{n_1},\ZZ_{n_1}\setminus\{0\})\times\dots\times\Cay(\ZZ_{n_r},\ZZ_{n_r}\setminus\{0\})\\
\cong&\Si_0\times \K_{n_1}\times\dots\times \K_{n_r}\\
=&\Si_0\times \K_{m_1}\times\dots\times \K_{m_s}\\
=&\Ga_2.
\end{align*}
Thereby we obtain
\begin{align*}
&\Cay(\ZZ_{n_0}\times\ZZ_{n_1}\times\dots\times\ZZ_{n_r},S\times(\ZZ_{n_1}\setminus\{0\})\times\dots\times(\ZZ_{n_r}\setminus\{0\}))\\
\cong&\Ga_1\cong\Ga_2
\cong\Cay(\ZZ_{n_0}\times\ZZ_{n_1}\times\dots\times\ZZ_{n_r},T\times(\ZZ_{n_1}\setminus\{0\})\times\dots\times(\ZZ_{n_r}\setminus\{0\})),
\end{align*}
whence there exists $\varphi\in\Aut(\ZZ_{n_0}\times\ZZ_{n_1}\times\dots\times\ZZ_{n_r})$ such that
\[
(S\times(\ZZ_{n_1}\setminus\{0\})\times\dots\times(\ZZ_{n_r}\setminus\{0\}))^\varphi
=T\times(\ZZ_{n_1}\setminus\{0\})\times\dots\times(\ZZ_{n_r}\setminus\{0\}).
\]
As $n_0,n_1,\dots,n_r$ are pairwise coprime, we conclude that $S^{\varphi_0}=T$ for some $\varphi_0\in\Aut(\ZZ_{n_0})$. This implies that $\Ga_0=\Cay(\ZZ_{n_0},S)\cong\Cay(\ZZ_{n_0},T)=\Si_0$, as desired.
\end{proof}

\vskip0.1in
\noindent\textit{Acknowledgements.} This project was initiated during the first named author's visit to the University of Melbourne, and was partially supported by National Natural Science Foundation of China (NNSFC~11771200 and~11931005). The authors would like to thank the anonymous referee for careful reading and valuable suggestions on this paper.

\end{document}